\documentclass[final]{amsart}
\usepackage{graphicx, amsmath, amssymb, amsfonts, verbatim,stmaryrd}
\usepackage{fancybox}
\usepackage[style=alphabetic,url=false, isbn=false, doi=false,maxbibnames=99]{biblatex}
\renewbibmacro{in:}{}
\AtEveryBibitem{
    \clearfield{urlyear}
    \clearfield{urlmonth}
    \clearfield{month}
    \clearfield{day}    
    }
\usepackage[margin=3cm,footskip=25pt,headheight=20pt]{geometry}

\usepackage{xcolor,nicefrac}
\usepackage[nomargin,inline]{fixme}
\makeatletter
\renewcommand*\FXLayoutInline[3]{{\@fxuseface{inline}
  \ignorespaces\noindent \ovalbox{\hspace{.01\textwidth} \begin{minipage}{.95\textwidth}
  	#3 \fxnotename{#1}: #2
  \end{minipage}\hspace{.01\textwidth}}}
  \newline}
\makeatother

\usepackage{hyperref}

\usepackage[capitalize]{cleveref}

\newtheorem{thm}{Theorem}

\newtheorem{lem}[thm]{Lemma}

\newtheorem{defn}[thm]{Definition}

\newtheorem{itheorem}{Theorem}

\newcommand{\N}{\mathbb{N}}
\newcommand{\Z}{\mathbb{Z}}
\newcommand{\C}{\mathbb{C}}
\newcommand{\A}{\mathcal{A}}
\newcommand{\cC}{\mathcal{C}}
\newcommand{\D}{\mathcal{D}}
\newcommand{\CPt}{\mathcal{C}_{P,t}}
\newcommand{\DPt}{\mathcal{D}_{P,t}}
\newcommand{\half}{\tfrac{1}{2}}

\newcommand{\nhalf}{\nicefrac{1}{2}}

\newcommand{\rad}{\operatorname{rad}}

\newcommand{\lerchp}[3]{\Psi_{+,#2}^{(#1)}(#3)}
\newcommand{\lerchm}[3]{\Psi_{-,#2}^{(#1)}(#3)}

\title[Degenerate twisted traces]{Degenerate twisted traces on\\ quantized Kleinian singularities of type A}
\author{Zev Friedman}
\email{zev.friedman@mail.utoronto.ca}
\address{Department of Pure Mathematics, University of Waterloo, Waterloo, ON, Canada}
\curraddr{Department of Pure Mathematics, University of Toronto, Toronto, ON, Canada}
\author{Ben Webster}
\email{ben.webster@uwaterloo.ca}
\address{Department of Pure Mathematics, University of Waterloo \&
Perimeter Institute for Theoretical Physics, Waterloo, Canada, ON
}

\begin{document}
\begin{abstract}
    We study the space of nondegenerate traces on quantized Kleinian singularities of type A by studying their complement, the degenerate traces.  In particular, we find the dimension of the space of twisted traces as a function of the corresponding automorphism and the quantization parameters, encoded in a polynomial $P$.
\end{abstract}
\maketitle

\section{Introduction}
Given a $\mathbb{C}$-algebra $\A$ and an automorphism $\omega$, a {\bf twisted trace} on $\A$ is a $\C$-linear map $T \colon \A\to \C$ such that $T(ab)=T(\omega(b)a)$.  
The topic of twisted traces on noncommutative algebras, especially quantizations of symplectic singularities, has attracted considerable mathematical attention in recent years.  The traces of interest to us have arisen in the study of conformal field theory \cite{beemDeformationQuantization2017}, but have raised mathematical questions of independent interest.  As shown in \cite{etingofShortStarProducts2020}, a particularly important role in the construction of short star products is played by twisted traces which are strongly nondegenerate, that is, where for all $a$, there is a $b$ with $\deg b\leq \deg a$ such that $T(ab)\neq 0$ (see Definition \ref{def:twisted-trace}).  In this paper, we study a weaker form of nondegeneracy, where we remove the condition $\deg b\leq \deg a$, which is easier to study.  Throughout, we use {\bf nondegenerate} to mean a trace satisfying this weaker condition.

In this paper, we consider the quantizations $\A_P$ of Kleinian singularities of type A and study their space of twisted traces for a natural family of automorphisms $g_t$ depending on a parameter $t\in \C\setminus\{0\}$.  This algebra depends on a choice of monic polynomial $P$, whose degree $d=\deg P$ fixes the singularity that we quantize.  These algebras have already been considered in \cite{etingofTwistedTraces2021}, but primarily with a focus on twisted traces on even quantizations and the question of positive definiteness.  We will consider these questions from a more elementary, algebraic perspective, focusing on the question: What is the set of nondegenerate twisted traces of $\A_P$ as a function of $P$?

As is often the case, we can answer this question by studying the complement of the nondegenerate traces, which are naturally the {\bf degenerate traces}. Since the space of all twisted traces is known to be a $d-1$ or $d$-dimensional $\C$-vector space (when $t=1$ or $t\ne 1$ respectively) and degenerate traces form a subspace, the only interesting invariant determining this structure is the dimension of the space of degenerate traces. 

\begin{itheorem}\label{th:A}
    The space of degenerate traces is a subspace whose dimension is $0\leq \delta(P)\leq d-1$ as defined in \eqref{eq:delta-def}, with every integer in this range being realized for every $t$.  In particular, we have:
    \begin{itemize}
        \item $\delta(P)=0$ if and only if $P$ has no pair of distinct roots whose difference is an integer.
        \item $\delta(P)=d-1$ if and only if all roots of $P$ are distinct and belong to a single coset $a+\Z$ for $a\in \C$.
    \end{itemize}
    In particular, the algebra $\A_P$ has a nondegenerate trace if $t\neq 1$.  On the other hand, when $t=1$, a nondegenerate trace exists if and only if $\delta(P)<d-1$, that is, $P$ has a multiple root or two roots whose difference is not integral.  
\end{itheorem}

We prove this by studying the {\bf formal Stieltjes transform} $F_T$ of a trace $T$.  This is a power series which encodes the information of a twisted trace.  We show below that this power series is the Taylor expansion at $z=\infty$ of a rational function if and only if $T$ is degenerate (\Cref{rational=degenerate}).  Thus, we prove Theorem~\ref{th:A} by characterizing the rational functions which appear in this way (Lemma~\ref{lem:Delta=0}).  This allows us to prove a general characterization of the dimension of the space of degenerate traces (\Cref{th:D-dimension}).  We will also discuss how similar tools can be applied to study strong nondegeneracy; this is a nonlinear condition, so it is much more complicated to study, but it can be characterized by the Pad\'e approximants of $F_T(x)$.  Nondegeneracy will fail on polynomials of degree $\leq n$ if the $n$th Pad\'e approximant to $F_T(x)$ has denominator of lower than expected degree (Lemma \ref{lem:Pade}).  

There are many possible extensions of this work, for example, considering general hypertoric enveloping algebras.  The space of twisted traces for these algebras is characterized by the ``quantum Hikita conjecture'' \cite{kamnitzerQuantumHikita2021} of Kamnitzer, McBreen, and Proudfoot (a theorem in the hypertoric case). These have been studied in more generality by the second author, Gaiotto, Hilburn, Redondo-Yuste, and Zhou \cite{gaiottoTwistedTraces2025}.  However, no consideration has been given there to understanding which of these traces are degenerate.

In addition to the proof of Theorem~\ref{th:A}, we discuss a few other interesting details about these traces.  In particular, we show how to write any degenerate trace as a sum of pullbacks of degenerate traces from algebras where $P$ has two roots and characterize the degenerate traces that appear concretely as twisted traces of finite-dimensional representations (Lemma~\ref{lem:finite-dimensional}).  Finally, we point out the existence of a meromorphic function satisfying the recursion \eqref{eq:tilde-recursion} expected of the formal Stieltjes transform of a trace, though the transform is a power series with radius of convergence $0$ and thus does not define a meromorphic function.

\subsection*{Acknowledgements}

The authors appreciate useful conversations with Pavel Etingof, Davide Gaiotto, and Daniil Klyuev about twisted traces.
B. W. is supported by an NSERC Discovery Grant. This research was supported in part by a Mathematics Undergraduate Research Award from the Faculty of Mathematics, University of Waterloo, and in part by Perimeter Institute for Theoretical Physics. Research at Perimeter Institute is supported by the Government of Canada through the Department of Innovation, Science and Economic Development Canada and by the Province of Ontario through the Ministry of Research, Innovation and Science.

\section{Background}

Given a nonconstant polynomial $P\in\C[x]$, the quantized Kleinian singularity $\A_P$ is the free algebra on $\C$ in three variables $u, v, z$, with the following relations:
$$zu=uz-u,\hspace{5mm} zv=vz+v,\hspace{5mm} uv=P(z-\half),\hspace{5mm} vu=P(z+\half).$$
We will write $\A=\A_P$ for brevity unless the identity of $P$ is crucial for a given result.  This algebra makes a number of appearances in different aspects of the theory of symplectic singularities.  It is the $W$-algebra associated to a subregular nilpotent of $\mathfrak{sl}_d$, a hypertoric enveloping algebra for a rank $d-1$ torus acting on a $d$-dimensional vector space, and the spherical Cherednik algebra for $\Z/d\Z$ acting on $\C$ by the $d$th roots of unity.  

Notice that $\A$ decomposes as a sum of eigenspaces of $\text{ad}z$, where $\text{ad}z$ maps $a\in\A$ to $za-az$, and that the eigenvalues are exactly the integers. For each $k\in\Z$, let $\A_k$ be the eigenspace of $\text{ad}z$ for the eigenvalue $k$. Note that $\A_0=\C[z]$ and that if $a\in\A_j,b\in\A_k$ then $ab\in\A_{j+k}$.

For all $t\in\C\setminus\{0\}$, let $g_t:\A\to\A$ be the automorphism given by $u\mapsto t^{-1}u$, $v\mapsto tv$, and $z\mapsto z$.  
\begin{defn}\label{def:twisted-trace}
A twisted trace on $\A$ with respect to $g_t$ is a linear functional $T:\A\to\C$ with the property $T(ab)=T(g_t(b)a)$ for all $a,b\in\A$.

Consider the {\bf radical} $\rad T$ of a twisted trace $T$ 
\[\rad(T)=\{a\in \A \mid T(ab)=0\text{ for all }b\in\A\}\]
The trace $T$ is {\bf degenerate} if $\rad(T)\neq 0$, that is,  if there exists $0\ne a\in \A$ such that $T(ab)=0$ for all $b\in\A$.  Otherwise, we have $\rad(T)=0$ and we call $T$ {\bf nondegenerate}.  

We will also be interested in a stronger version of nondegeneracy: we call $T$ {\bf $n$-degenerate} if for some nonzero $p(z)$ of degree $\leq n$, we have $T(p(z)q(z))=0\text{ for all }q\in \C[x]$ satisfying $\deg q\leq n$.  
\end{defn}

We can also think about this nondegeneracy in terms of the infinite Hankel matrix
\begin{equation}\label{eq:Hankel}
    H=\begin{bmatrix}
    T(1) & T(z) & T(z^2) &\cdots \\
    T(z) & T(z^2) & T(z^3) &\cdots \\
    T(z^2) & T(z^3) & T(z^4) &\cdots \\
    \vdots & \vdots &\vdots & \ddots 
    \end{bmatrix}
\end{equation}
The trace $T$ is nondegenerate if the matrix $H$ is nondegenerate and $n$-degenerate if its principal $n\times n$ minor is degenerate.  

 Let $\CPt$ denote the space of twisted traces on $\A_P$ with respect to $g_t$, and let $\DPt$ denote the set of degenerate twisted traces on $\A_P$ with respect to $g_t$.  One can readily verify that $\CPt$ has a natural vector space structure and that $\DPt$ is a vector subspace, since if $T(ab)=0$ and $T'(a'b)=0$ for all $b$, then $(T+T')(aa'b)=0$ as well.

 Note that the twisted trace property shows that $\rad T$ is a two-sided ideal, since if $a\in \rad T$ and $a',b\in \A$, we have \[T((aa')b)=T(a(a'b))=0 \qquad \qquad T((a'a)b)=T(a(bg_t^{-1}(a')))=0.\]

If $P_1(x)\mid P_2(x)\in\C[x]$ and we choose $Q_1,Q_2$ so that $P_1(x)Q_1(x)Q_2(x)=P_2(x)$, then there is a natural homomorphism $\varphi_{Q_1,Q_2}\colon \A_{P_2}\to \A_{P_1}$ defined by
\begin{equation}\label{varphi-def}
\varphi_{Q_1,Q_2}(u)=Q_1(z-\half)u \qquad \varphi_{Q_1,Q_2}(v)=Q_2(z+\half)v\qquad \varphi_{Q_1,Q_2}(z)=z.    
\end{equation}
Note that if $aP_1(x)=P_2(x)$ for some $0\ne a\in\C$, then the map $\varphi_{a,1}$ is an isomorphism. So for the remainder of the paper it suffices to consider $\A_P$ where $P$ is monic.
This map commutes with the action of $g_t$, so the pullback by $\varphi_{Q_1,Q_2}$ preserves twisted traces.  

It will also be occasionally useful to note that $\A_{P(x)}\cong \A_{P(-x)}$ via the isomorphism
\begin{equation}\label{eq:negate-z}
    u\mapsto v \qquad v\mapsto u \qquad z\mapsto -z
\end{equation}
Note that this isomorphism intertwines $g_t$ and $g_{t^{-1}}$, so pullback by it swaps twisted traces for these automorphisms.  
Similarly,  $\A_{P(x)}\cong \A_{P(x-r)}$ via 
\begin{equation}\label{eq:shift-z}
    u\mapsto u \qquad v\mapsto v \qquad z\mapsto z+r
\end{equation}

In their paper \cite{etingofTwistedTraces2021}, Etingof et al. showed that:
\begin{lem}[\mbox{\cite[Prop. 2.3]{etingofTwistedTraces2021}}]
    A linear functional $T$ on $\A$ is a twisted trace with respect to $g_t$ if and only if 
    \begin{enumerate}
        \item $T(\A_k)=0$ for all $k\ne0$, and
        \item $T(S(z-\half)P(z-\half))=tT(S(z+\half)P(z+\half))$ for all $S\in\C[x]$.
    \end{enumerate}  
\end{lem}
This means that a twisted trace is uniquely determined by its behaviour on $\A_0$.

\begin{defn}\label{def:FST}
If $T$ is a linear functional on $\C[z]$, then the {\bf formal Stieltjes transform} of $T$ is the formal power series $F_T\in x^{-1}\C[[x^{-1}]]$ given by 
$$F_T(x)=\sum_{n=0}^\infty x^{-n-1}T(z^n)=T(\tfrac{1}{x-z}).$$
\end{defn}  

For $n\in\Z$, let $\C_n[x]$ denote the vector space of polynomials of degree at most $n$ (in particular, $\C_n[x]=\{0\}$ for $n<0$).

\begin{lem}[\mbox{\cite[Prop. 2.5]{etingofTwistedTraces2021}}]\label{FS-injective}
    The map $T\mapsto P(x)(F_T(x+\half)-tF_T(x-\half))$ is a vector space isomorphism  $\CPt\overset{\sim}{\to}\C_{\deg P-1}[x]$ if $t\ne1$, and from $\CPt\overset{\sim}{\to}\C_{\deg P-2}[x]$ if $t=1$.
    In particular, if $P(x)(F(x+\half)-tF(x-\half))$ is a polynomial, then $F(x)$ must be the formal Stieltjes transform of some twisted trace on $\A$.
\end{lem}

Note that this implies that $F_T(x+\half)-tF_T(x-\half)$ is a rational function which has a pole at $a$ of order at most the vanishing order $p_a$ of $P$ at $x=a$.
Thus, we have the principal parts expansion
    \begin{equation}\label{eq:P-principal}
      F_T(x+\half)-tF_T(x-\half) = \sum_{a\in \C} \frac{D_{a}^{(1)}}{x-a} +\cdots+  \frac{D_{a}^{(p_a)}}{(x-a)^{p_a}}.  
    \end{equation}
\section{Degenerate traces}

\begin{lem}
    A twisted trace $T$ is degenerate on $\A$ if and only if it is degenerate on $\C[z]\subset\A$.
\end{lem}
\begin{proof}
    $(\Rightarrow)$ Suppose that $T$ is degenerate on $\A$. 
    Then there exists some nonzero $a\in \A$ such that $T(ab)=0$ for all $b\in\A$. 
    Since $\A=\oplus_{k\in\Z}\A_k$, we may write $a=\sum_{k=-K}^K a_k$ for $a_k\in\A_k$, and $K\in\N$. Then since $a\ne0$, there exists $j\in\Z$ such that $a_j\ne0$. 
    For all $b\in \A_{-j}$, we have $T(ab)=\sum_{k=-K}^K T(a_kb)=T(a_jb)$ since $a_kb\in\A_{k-j}$. Choose some nonzero $c\in \A_{-j}$ and note that $0\ne a_jc\in\A_0=\C[z]$. Then for all $R(z)\in\C[z]$, we have that $cR(z)\in\A_{-j}$ and so $T(a_jcR(z))=T(acR(z))=0$, so $T$ is degenerate on $\C[z]$.

    $(\Leftarrow)$ Suppose $T$ is degenerate on $\C[z]$. Then there exists some nonzero $S(z)\in\C[z]$ such that $T(S(z)R(z))=0$ for all $R\in\C[x]$. Then for all $b\in\A_0$ we have $T(S(z)b)=0$, and for all $b\in\A_k$ with $k\ne0$, we have $S(z)b\in\A_k$ and so $T(S(z)b)=0$. Since $\A$ is the direct sum of these spaces, we have $T(S(z)b)=0$ for all $b\in\A$. 
\end{proof}

Note that $I_T=\rad T\cap \C[z]$ is an ideal, so it is generated by a unique monic polynomial $S(x)$, which is characterized as the unique monic polynomial of minimal degree in $I_T$.

\begin{thm}\label{rational=degenerate}
    We have an equality 
    \[I_T=\{S(z) \mid F_T(z)S(z)\in \C[z]\}\]
    In particular, a twisted trace $T$ on $\A$ is degenerate if and only if $F_T(x)$ is a rational function of $x$.  In this case, if we write $F_T(x)=\frac{R(x)}{S(x)}$ with $S$ monic and $\gcd(R,S)=1$, then $S$ generates $I_T$.
\end{thm}
\begin{proof}
    Suppose that $S(z)\in I_T$. Then we have 
    $$S(x)F_T(x)=S(x)F_T(x)-\sum_{n=0}^\infty x^{-n-1}T(S(z)z^n)=T(\tfrac{S(x)-S(z)}{x-z})$$
    Let $R(x)=T(\tfrac{S(x)-S(z)}{x-z})$, and note that $R(x)$ is a polynomial. Then $S(x)F_T(x)=R(x)$ as desired.

Now suppose that $R,S\in\C[x]$ satisfy $S(x)F_T(x)=R(x)$. Then 
    $$R(x)=S(x)F_T(x)=T(\tfrac{S(x)}{x-z})=T(\tfrac{S(x)-S(z)}{x-z})+T(\tfrac{S(z)}{x-z})$$
    Since every term in $R(x)$ and $T(\tfrac{S(x)-S(z)}{x-z})$ has a nonnegative exponent on $x$ and every term in $T(\tfrac{S(z)}{x-z})$ has a negative exponent on $x$, in order for the two sides to be equal we must have
    $$0=T(\tfrac{S(z)}{x-z})=\sum_{n=0}^\infty x^{-n-1}T(S(z)z^n),$$
    and so $T(S(z)b)=0$ for all $b\in\C[z]$ since $\{z^n\}_{n=0}^\infty$ is a basis for $\C[z]$.  Thus, we have $S(z)\in I_T$.  This completes the proof.  
\end{proof}
Note that since the Laurent expansion of $F_T(x)$ at $x=\infty$ (that is, its expansion as a power series in $\C[[x^{-1}]]$) contains no nonnegative powers of $x$, we must have $\deg R<\deg S$.  Thus, we can expand $R/S$ in its principal parts expansion---if we let $m_a$ be the order of vanishing of $S$ at $x=a$ for $a\in \C$, then we have
\[\frac{R(x)}{S(x)}=\sum_{a\in \C} \left( \frac{C_a^{(1)}}{x-a}+\cdots + \frac{C_a^{(m_a)}}{(x-a)^{m_a}}\right)\qquad C_a^{(k)}=\operatorname{res}_{x=a}\left((x-a)^{k-1}\frac{R(x)}{S(x)}\right)\]
Note that this principal part decomposition is equivalent to the Vandermonde decomposition of a Hankel matrix $H$ defined in \eqref{eq:Hankel}.  For simplicity, assume $F_T(x)=\Pi_a(x)$ for a single $a\in \C$.  In this case, $T(z^k)=\sum_{r=1}^{\infty}C_a^{(r)}\binom{r-1}{k}a^{k-r+1}$ by the usual Taylor expansion
\[ \frac{1}{(x-a)^r}=\frac{1}{x^r}+a\binom{r-1}{r} \frac{1}{x^{r+1}} + a^2\binom{r-1}{r+1} \frac{1}{x^{r+2}}+\cdots \]
at $x=\infty$.  Let $r$ be the order of the pole of $\Pi_a(x)$.  By \cite[Th. 2]{boley1997general}, we can factor the Hankel matrix as a product $H=(V^{(a)})^TD^{(a)}V^{(a)}$ where $V$ is an $r\times \infty$ matrix with $V_{ij}^{(a)}=(\binom{j-1}{i-1}a^{j-i})$ and $D$ is the matrix $D_{ij}^{(a)}=(C^{(i+j-1)}_a)$. One can readily calculate that the $i,j$ entry of $(V^{(a)})^TD^{(a)}V^{(a)}$ is
\[ \sum_{k,\ell\geq 1} a^{i+j-k-\ell}C^{(k+\ell-1)}_a \binom{k-1}{i-1}\binom{\ell-1}{j-1}=\sum_{s\geq 1}a^{i+j-s-1}C^{(s)}\binom{s-1}{i+j-2}=T(z^{i+j-2})\] using the reindexing $s=k+\ell-1$ and the usual relation on binomial coefficients.    The general case is found by stacking the matrices $V^{(a)}$ for different $a\in \C$ and taking the matrices $D^{(a)}$ as diagonal blocks, as in \cite[Th. 2]{boley1997general}.

By \eqref{eq:P-principal}, the coefficients $C_{a}^{(k)}$ satisfy the relations
\begin{equation}\label{eq:C-D}
    C_{a-1/2}^{(k)}-tC_{a+1/2}^{(k)} =D_a^{(k)}.
\end{equation}

Consider $r\in\C$ and $P\in\C[x]$; if $r$ is a root of $P$ of order $n$, we will say that $r$ is an $n$-fold root of $P$.  
For each complex number $a$, we let $n_a^+$ be the largest integer such that there is a root of order $\geq n_a^+$ of $P$ in the set $a+\frac{1}{2}+\Z_{\geq 0}$ and similarly with $\geq n_a^-$ and the set $a-\frac{1}{2}-\Z_{\geq 0}$.   Let $n_a=\min (n_a^+,n_a^-)$.
\begin{lem}  For a fixed polynomial $P$ and any degenerate trace $T$, we have $m_a\leq n_a$ for all $a\in \C$.  
\end{lem}
\begin{proof}
We need only prove that $n^+_a\geq m_a$.   Applying the same argument after the isomorphism \eqref{eq:negate-z} will imply that $n_a^-\geq m_a$ as well.
Restrict consideration to the set $a+\Z$. Since $P$ and $S$ have finitely many zeros, we have $m_{a'}=n_{a'}^+=0$ for all but finitely many $a'\in a+ \Z_{\ge0}$. Thus, we can apply induction with base case being some $a$ such that $m_{a'}=n_{a'}^+=0$ for all $a'\ge a$. For the inductive step, assume that $m_{a+1}\le n_{a+1}^+$.  

    Recall that $P(x)(F_T(x+\half)-tF_T(x-\half))\in\C[x]$, and observe that
    $$F_T(x+\half)-tF_T(x-\half)=\frac{R(x+\half)}{S(x+\half)}-t\frac{R(x-\half)}{S(x-\half)}
.$$
Consider this rational function at $x=a+\half$.  It can only have a pole if $P(a+\half)=0$, and the order $p$ of its pole is at most the vanishing order of $P$ at this point.  Let us compare this with the vanishing orders $m_{a+1}$ and $m_a$; since $R$ and $S$ are coprime, these are also the order of the pole of $R/S$ at these points.  

If $p<m_{a}$, then by the nonarchimedean triangle inequality for vanishing orders, we must have that $m_{a+1}=m_a$. By induction, we know that $n_a^+\geq n_{a+1}^+\geq m_{a+1}=m_a$.  

On the other hand, if $p\geq m_a$, then $b=a+\half$ is the desired root of order $\geq m_a$.  This completes the proof.
\end{proof}

Thus, we can write the formal Stieltjes transform of any degenerate twisted trace as:
\[\frac{R(x)}{S(x)}=\sum_{a\in \C} \Pi_a(x) \qquad \Pi_a(x) =\left( \frac{C_a^{(1)}}{x-a}+\cdots + \frac{C_a^{(n_a)}}{(x-a)^{n_a}}\right)\]

Furthermore, since $P(x)(F_T(x+\half)-tF_T(x-\half))\in\C[x]$, it has vanishing principal parts at all $a\in \C$.

\begin{lem}\label{Stieltjes-change}
    The rational function $\frac{R(x)}{S(x)}$ is the formal Stieltjes transform of a degenerate twisted trace if and only if for all $a\in \C$ which are not roots of $P$
    \begin{equation}\label{eq:t-condition}
        C_{a-1/2}^{(\ell)}=tC_{a+1/2}^{(\ell)}
    \end{equation}
for all $\ell\leq n_{a\pm 1/2}$, and 
$$C_{a-1/2}^{(\ell)}=C_{a+1/2}^{(\ell)}=0$$ 
for all $\ell>n_{a\pm 1/2}$.
\end{lem}
We can also write this condition in terms of the coefficients $D_{a}^{(\ell)}$ defined by \eqref{eq:P-principal}. Applying \eqref{eq:C-D}, we have  
\begin{equation}\label{eq:C-D2}
    C_{a}^{(k)}=( C_{a}^{(k)}-tC_{a+1}^{(k)})+t ( C_{a+1}^{(k)}-tC_{a+2}^{(k)})+\cdots =\sum_{j=0}^{\infty} t^jD_{a+j+1/2}^{(k)}
\end{equation}
Since $C_{a}^{(k)}=0$ for all but finitely many $a$, we find that:
\begin{lem}\label{lem:Delta=0}
    The twisted trace with the formal Stieltjes transform $F_T(x)$ is degenerate if and only if
\[\Delta^{(k)}_a=\sum_{j=-\infty}^{\infty} t^jD_{a+j+1/2}^{(k)} =0\]
for all $a\in \C$ and $k$.  
\end{lem}
Note that $\Delta^{(k)}_a=t\Delta^{(k)}_{a-1}$, so the vanishing of this quantity is really a condition on the coset of $a$ modulo $\Z$.  
This description makes it clear how large the space of degenerate traces is. For each coset $[a]\in \C/\Z$, let $\delta_{[a]}(P)$ be the number of roots in $[a]$, counted with multiplicity, minus the maximum multiplicity of a root in this set. Let $\delta(P)$ be the sum of these statistics over $\C/\Z$.  That is:
\begin{equation}\label{eq:delta-def}
\delta(P)=\sum_{[a]\in \C/\Z}\delta_{[a]}(P)=\sum_{a'\in [a]} p_{a'}-\max_{a'\in [a]} p_{a'}
\end{equation}
\begin{thm}\label{th:D-dimension}
    $\dim \DPt=\delta(P)$.
\end{thm}
\begin{proof}
    For simplicity, we can assume that all the roots of $P$ lie in $\Z+\half$ and prove that $\dim \DPt=\delta_{[\nhalf]}(P)$.   Let $\displaystyle \pi=\max_{a'\in \Z+\nhalf} p_{a'}$.  The space of all possible choices of $D_a^{(\ell)}$ for $\ell=1,\dots, p_a$ and all $a\in \Z+\half$ is exactly $\sum_{a\in \Z+\nhalf} p_{a}$, and the statistics $\Delta^{(\ell)}_0=\sum_{j=-\infty}^{\infty} t^jD_{j+1/2}^{(\ell)}$ for $\ell=1,\dots, \pi$ are linearly independent functions. Thus, the dimension of their common kernel, which is the space of degenerate traces, is exactly $\delta_{[\nhalf]}(P)=\deg P-\pi$.  This completes the proof.
\end{proof}
This shows that the dimension of $\DPt$ can range anywhere from $0$ (if no distinct roots of $P$ differ by an element of $\Z$) to $\deg P-1$ (if all roots are simple and lie in the same coset of $\Z$).  

More explicitly, we see that we have coordinates on $\DPt$ given by the functions $D^{(k)}_{a}$ where $k$ ranges over integers such that $P$ vanishes to order $\geq k$ at $a$ and at $a+j$ for some $j\in \Z_{>0}$.

Now we discuss how we can compare the traces for different values of $P$.  
We have introduced an algebra homomorphism $\varphi_{Q_1,Q_2}\colon \A_{PQ_1Q_2}\to \A_{P}$ and considered pullback of twisted traces under it.  Since this map is the identity on $\C[z]$, it preserves the formal Stieltjes transform.    Since we can easily tell if a trace is 0 or degenerate from its formal Stieltjes transform, this shows:
\begin{thm}
    If $P_1(x)\mid P_2(x)\in\C[x]$, then the pullback by the homomorphism $\varphi_{Q_1,Q_2}$ introduced in \eqref{varphi-def} induces an injective map $\cC_{P_1,t}\hookrightarrow \cC_{P_2,t}$. This homomorphism preserves the formal Stieltjes transform.  The subspace $\D_{P_1,t}$ is the preimage of $\D_{P_2,t}$.  
\end{thm}
\begin{proof}
    Since the map $\varphi=\varphi_{Q_1,Q_2}$ acts by the identity on $\C[z]$, it preserves formal Stieltjes transforms: $F_T=F_{T\circ \varphi}$.  Thus, by Lemma~\ref{FS-injective}, pullback sends nonzero traces to nonzero traces, and thus is injective.  Similarly, by \cref{rational=degenerate}, the pullback of a trace is degenerate if and only if the original trace is.  
\end{proof}

Note that if $P_1\ne P_2$ then this map is never an isomorphism on the space of all traces, but it may be on the space of degenerate traces, which we can check by seeing if the dimensions are the same.  This happens when $\delta(P_1)=\delta(P_2)$.  This condition can be stated in a number of ways, but one way to say it is that if $a$ is a zero of different orders $p<p'$ for $P_1$ and $P_2$, then all other zeros in $a+\Z$ must have the same vanishing order $\leq p$ for $P_1$ and $P_2$.

Employing different decompositions of the rational function $F_T(x+\half)-tF_T(x-\half)$, we obtain different decompositions of our twisted trace, which are often interestingly compatible with this pullback.

For example, a rational function with a single pole of order $s$ at $a$ will give a twisted trace for $P(x)=(x-a)^s$, and as long as $t\neq 1$, we can write any twisted trace canonically as the sum of pullbacks of these traces, corresponding to the principal parts decomposition of $F_T(x+\half)-tF_T(x-\half)$.  However, this process is badly incompatible with degeneracy, since when $P$ has a single root, all nonzero traces will be nondegenerate.

More natural for us is to decompose according to the order of poles.  That is, we let:
\[D^{(k)}(x)=\sum_{a\in \C}\frac{D_a^{(k)}}{(x-a)^k}.\]
Of course, $F_T(x+\half)-tF_T(x-\half) =\sum_{k=1}^{\infty}D^{(k)}(x)$.  
\begin{lem}
    We have a unique decomposition $T=\sum_{k=1}^{\infty}T^{(k)}$ into twisted traces pulled back from $\A_{P^{(k)}}$ where $P^{(k)}=\prod (x-a)^{\min(p_a,k)}$, with the property that every term in the partial fraction expansion of $F_{T^{(k)}}$ has only poles of order $k$. The trace $T$ is degenerate if and only if $T^{(k)}$ is degenerate for all $k$.
\end{lem}
\begin{proof}
    We can characterize these traces uniquely by $F_{T^{(k)}}(x+\half)-tF_{T^{(k)}}(x-\half)= D^{(k)}(x)$.  If $t\neq 1$, any rational function which only has poles at the right points corresponds to a twisted trace.  If $t=1$, we need the additional condition $\sum_{a\in\C}D^{(1)}_a=0$.  This holds for $D^{(1)}$ if and only if it holds for $F_T(x+\half)-tF_T(x-\half)$, and is automatic for $D^{(k)}$ for $k>1$.  
    
    The condition $\Delta_a^{(k)}=0$ holds for $D^{(k)}(x)$ if and only if it does for $F_T(x+\half)-tF_T(x-\half)$, so Lemma~\ref{lem:Delta=0} shows that $T$ is degenerate if and only if $T^{(k)}$ is degenerate for all $k$.  
\end{proof}

Notice that we can also uniquely characterize the twisted traces in the previous lemma by simply taking the partial fraction expansion of $F_T$ and extracting the terms whose poles are of order $k$ to be $F_{T^{(k)}}$, and this implies that $T^{(k)}$ is a degenerate twisted trace on $P^{(k)}$.

There are various possibilities for decomposing further: for example, for $a\in\C$ and $j\in\Z_{\ge0}$, if $a,a+j$ are roots of order $\geq k$ of $P$ with no roots of order $\geq k$ in $a+1,\dots, a+j-1$, the sum \[\frac{C_{a+1/2}^{(k)}}{{(x-a-1/2)}^k}+\cdots + \frac{C_{a-j-1/2}^{(k)}}{{(x-a-j+1/2)}^k} \] is the Stieltjes transform of a degenerate twisted trace for $P=(x-a)^k(x-a-j)^k$, and $T^{(k)}$ can be written uniquely as the sum of these traces.  

Alternatively, if $a$ and $a+j$ are as above, then we have a degenerate trace $T'$ with $F_{T'}(x+\half)-tF_{T'}(x-\half)=\sum_{i=1}^k D_a^{(k)}\left(\frac{1}{(x-a)^i}-\frac{t^{-j}}{(x-a-j)^i}\right)$ such that $T''=T-T'$ is a degenerate twisted trace such that $F_{T''}(x+\half)-tF_{T''}(x-\half)$ has trivial principal part at $x=a$.  Applying this inductively, we can write any degenerate trace as a sum of traces pulled back from $\A_{P'}$ where $P'$ has two roots with integer difference.

\section{Strong nondegeneracy}

We now consider how similar ideas can be applied to study $n$-degeneracy.  
A useful fact to note is that any formal series $F_T(x)\in x^{-1}\C\llbracket x^{-1}\rrbracket$ has a unique $[n-1/n]$-Pad\'e approximant at $x=\infty$ 
\[\frac{R_{n-1}(x)}{S_n(x)}\approx F_T(x)\]
for $S_n$ a monic polynomial of degree $\leq n$ and $\deg R_{n-1} <\deg S_n$.  
This rational function is uniquely characterized by the equivalent conditions
\[F_T(x)-\frac{R_{n-1}(x)}{S_n(x)}\in x^{-2n-1}\C\llbracket x^{-1}\rrbracket  \qquad S_n(x)F_T(x)-R_{n-1}(x)\in x^{-n-1}\C\llbracket x^{-1}\rrbracket\]  These conditions show that $R_{n-1}(x)$ is determined by $S_n(x)F_T(x)$, so the polynomial $S_n(x)$ is determined by the condition that the $x^{-1-k}$ coefficients of $S_n(x)F_T(x)$ vanish for $k=0,\dots, n-1$.  These coefficients can also be interpreted as the traces $T(S_n(z)z^k)$.
That is, $S_n(x)$ is the unique polynomial of minimal degree such that $S_n(x)$ is orthogonal to all polynomials of degree $< n$.  There must be such a polynomial of degree $n$ for dimension reasons, but it might be that in some cases, there is one of lower degree, in which case $T$ is $(n-1)$-degenerate.  That is:
\begin{lem}\label{lem:Pade}
    The trace is $n$-degenerate if and only if $\deg S_{n+1}\leq n$, that is, if $S_{n+1}=S_n$ or equivalently $R_{n}/S_{n+1}=R_{n-1}/S_n$.  
\end{lem}
For any value of $n$, it is easy to construct examples of nondegenerate traces which are nonetheless $n$-degenerate:  For any fixed $\frac{R}{S}$ with $\deg R<\deg S<n$, we can consider the twisted trace with $F_T(x)=\frac{R(x+\half)}{S(x+\half)}-t\frac{R(x-\half)}{S(x-\half)}+\frac{1}{(x-a)^k}$ for some $k\gg n$, which has $S_n=S_{n+1}=\cdots =S_{k-2}=S$, but this trace is nondegenerate.   

It seems to be a very complicated problem to compute exactly which traces are $n$-degenerate for a fixed $n$.  We hope that some progress can be made from the observation that:
\begin{align}
    \label{approx-recursion1}
    \frac{R_{n-1}(x+\half )}{S_n(x+\half )}-t\frac{R_{n-1}(x-\half )}{S_n(x-\half )}&\in \frac{Q(x)}{P(x)}+x^{-2n-1}\C\llbracket x^{-1}\rrbracket & t&\neq 1\\
    \label{approx-recursion2}
        \frac{R_{n-1}(x+\half )}{S_n(x+\half )}-t\frac{R_{n-1}(x-\half )}{S_n(x-\half )}&\in \frac{Q(x)}{P(x)}+x^{-2n-2}\C\llbracket x^{-1}\rrbracket & t&=1
\end{align}
where $Q(x)=P(x)(F_T(x+\half)-tF_T(x-\half))\in\C[x]$.

\begin{thm}\label{Pade-degenerate}
    The polynomials $R_{n-1}(x)$ and $S_{n}(x)$ are the unique polynomials such that $\deg R_{n-1}<\deg S_n\leq n$, $\gcd(R_{n-1},S_n)=1$, $S_n$ is monic and \eqref{approx-recursion1} or 
    \eqref{approx-recursion2} holds, depending on whether $t\neq 1$ or $t=1$.
\end{thm}
\begin{proof}
    If all these conditions hold for some polynomials $r,s$, then $\bar{F}(x)=F_T(x)-r(x)/s(x)$ is a power series that satisfies 
    \[\bar{F}(x+\half)-t\bar{F}(x-\half)=\frac{Q(x)}{P(x)}- \frac{r(x+\half )}{s(x+\half )}-t\frac{r(x-\half )}{s(x-\half )}\in x^{-2n-1}\C\llbracket x^{-1}\rrbracket. \]  If $t=1$, we can further conclude that $\bar{F}(x+\half)-t\bar{F}(x-\half)\in x^{-2n-2}\C\llbracket x^{-1}\rrbracket$.  

    Let $k+1$ for $k\geq 0$ be the order of vanishing of $\bar{F}(x)=a_{k}x^{-k-1}$ at $x=\infty$.  Taylor expanding $(x\pm \half)^{-k}=x^{-k}\pm \frac{k+1}{2}x^{-k-1}+\cdots$ at $x=\infty$, we find that \[\bar{F}(x+\half)-t\bar{F}(x-\half)=(t-1)a_kx^{-k-1}+((1-t)a_{k+1}-(k+1)a_k ) x^{-k-2}\] vanishes to order $k+1$ and order $k+2$ if $t=1$,  so we must have $k\geq 2n$.  This proves that $r(x)/s(x)=R_{n-1}(x)/S_n(x)$ is the expected Pad\'e approximant, and the conditions on these polynomials guarantee that $r(x)=R_{n-1}(x)$ and $s(x)=S_n(x)$.  
\end{proof}
It is a very interesting question how we might be able to find the principal parts expansion of $R_{n-1}(x)/S_n(x)$ from that of $\frac{Q(x)}{P(x)}$. At the moment, we see no practical way of doing this.

\section{Finite-dimensional modules}
One interesting source of degenerate twisted traces is finite-dimensional modules over $\A$.  Assume $M$ is such a finite-dimensional module and $\alpha\colon M\to M$ is a map satisfying 
\begin{equation}\label{eq:twisting}
    a\cdot \alpha(m)=\alpha(g_t(a)\cdot m)\qquad \text{for all}\quad a\in\A, m\in M.
\end{equation} Such an $\alpha$ is called a \textbf{twisting map}. We call the map $\A\to \C$ defined by $a\mapsto \operatorname{tr}(M,\alpha\circ a)$ the {\bf twisted trace of the module $M$ and map $\alpha$.}  This is manifestly a twisted trace in the sense of Definition~\ref{def:twisted-trace} since $\operatorname{tr}(M,\alpha\circ ab)=\operatorname{tr}(M,b\circ \alpha\circ a)=\operatorname{tr}(M,\alpha\circ g_t(b)a)$.  

\begin{lem}\label{lem:finite-dimensional}
    A degenerate twisted trace is a twisted trace of a finite-dimensional module if and only if $C_a^{(k)}=0$ for all $a$ and all $k>1$, or equivalently $D_a^{(k)}=0$ for all $a$ and all $k>1$.
    
    In particular, every degenerate trace is of this form if and only if $P$ has no multiple roots.
\end{lem}
\begin{proof}
    The finite-dimensional representations of this algebra are well-known: if $a,a+j$ are roots of $P$ with no roots in $a+1,\dots, a+j-1$, then there is a unique irreducible representation $S_{a,a+j}$ where the element $z$ has simple spectrum $\{a+\half,\dots, a+j-\half\}$, and all irreducible representations are of this form.   

    Given a finite-dimensional module $M$ over $\A$ and a twisting map $\alpha\colon M\to M$, since $\alpha$ and $z$ commute, they must have a simultaneous eigenvector $x$ with eigenvalues $b$ and $\lambda$.  Note that $ux,vx$ are also eigenvectors with eigenvalues $b\mp 1$ and $t^{\pm 1}\lambda $. Since $z$ has finite spectrum, we must have $v^mx=u^mx=0$ for some $m$. If for some $q$, we have $uv^qx\ne0$, we replace $x$ by $v^qx$ for the largest such $q$.  If no such $q$ exists, we replace $x$ by $u^sx$ where $s$ is maximal amongst integers such that $u^sx\neq 0$.  Thus, we assume that $ux=0$, and the assumptions above guarantee that $x$ generates a copy of $S_{b-1/2,b+j+1/2}$ for some $j$.  Furthermore, this submodule has a basis given by $\{x,vx,\dots, v^{j-1}x\}$, so it is invariant under $\alpha$.  Thus \[\operatorname{tr}(M,\alpha\circ a)= \operatorname{tr}(M/\A x,\alpha\circ a)+ \operatorname{tr}(\A x,\alpha\circ a)\] and we can reduce to assuming $M$ is irreducible, and $\alpha$ and $z$ have a basis of simultaneous eigenvectors.  We can explicitly calculate that \[F_T(x)=\sum_{n=0}^{\infty} x^{-n-1}\sum_{s=0}^{j-1}\lambda t^{-s}(b+s)^n= \sum_{s=0}^{j-1}\frac{\lambda t^{-s}}{x-a}.\] In the notation introduced earlier, this is the twisted trace corresponding to $C_b^{(1)}=tC_{b+1}^{(1)}=\cdots =t^jC_{b+j}^{(1)}=\lambda$ and all other coefficients 0.  This shows that a twisted trace of a finite-dimensional module must have this form and, in particular, has $C_a^{(k)}=0$ for $k>1$.
    
On the other hand, we have shown that a degenerate twisted trace with $C_a^{(k)}=0$ whenever $k>1$ is a sum of the traces above, and thus it can be written as a twisted trace on a sum of simples with an appropriate choice of $\alpha$. 
\end{proof}

It is possible to deal with the higher order poles by loosening the restriction \eqref{eq:twisting}.
Consider the case where 
\begin{align}
    F_T(x)&=C\sum_{j=0}^{n-1} \frac{t^j}{(x-a-j)^k}\label{one-FT}\\ 
    F_{T}(x+\half)-tF_{T}(x-\half)&=\frac{C}{(x-a+\half)^k} -\frac{Ct^n}{(x-a-n+\half)^k},\notag
\end{align} so $a-\half$ and $a+n-\half$ are both $k$-fold roots of $P$.  We can define a finite-dimensional representation $M$ with a basis $e_j^{(p)}$ for $j=0,\dots, n-1$ and $p=1,\dots, k$ (in all equations below, whenever an index is outside these ranges, $e_j^{(p)}=0$ by convention), with the action:
\begin{equation*}
    z\cdot e_j^{(p)}=(a+j)e_j^{(p)}+e_j^{(p+1)} \qquad u\cdot e_j^{(p)} =P(z-\half )e_{j-1}^{(p)} \qquad v\cdot e_j^{(p)}= e_{j+1}^{(p)}
\end{equation*}
This is the same as the representation generated by the element $e_0^{(1)}$ subject to the relations
\begin{equation*}
    (z-a)^k\cdot e_0^{(1)}=0\qquad u\cdot e_0^{(1)}=0 \qquad v^{n}\cdot e_0^{(1)}=0.
\end{equation*}
We can define a map $\alpha \colon V\to V$ sending $e_j^{(k)}\mapsto Ct^je_j^{(k-p+1)}$; note that $z\cdot\alpha(m)\neq\alpha(z\cdot m)$, so this does not satisfy \eqref{eq:twisting}.  Despite this, one can calculate that $T(a)=\operatorname{tr}(M,\alpha\circ a)$ is a twisted trace with $F_T(z)$ given by \eqref{one-FT}.  Summing over several representations, we can get any degenerate trace as $T(a)=\operatorname{tr}(M,\alpha\circ a)$ for some $(M,\alpha)$.  It's not clear to us in general how to test if a map $\alpha$ gives a twisted trace.

\section{Lerch transcendents}

A natural question raised by our calculations with Stieltjes transforms is whether we can interpret $F_T(x)$ as an analytic function for any nondegenerate traces.  In a certain sense, the answer is obviously not:
\begin{lem}
    If $F_T(x)$ is not the Taylor expansion of a rational function at $x=\infty$, then the radius of convergence (in terms of $x^{-1}$) of this power series is $0$.  
\end{lem}
\begin{proof}
    Assume not.  Then for some real number $\epsilon$, we have that $F_T(x)$ converges absolutely to a holomorphic function on $\{x\in \C\mid |x|>\epsilon\}$, and in fact on its closure in the Riemann sphere.  
    We can use \eqref{eq:P-principal} to meromorphically extend this function to the entirety of the Riemann sphere.  
    The argument of Lemma~\ref{Stieltjes-change} still applies and shows that this function can only have poles at points of the form $a+\half+\Z$ for $a$ a root of $P$ which lies inside $\{x\in \C\mid |x|\leq \epsilon\}$, with the order of these poles bounded by the vanishing order of $P$ at $x=a$.   
    There are only finitely many such points, which shows that $F_T(x)$ is a meromorphic function on the Riemann sphere with finitely many poles, all of finite order.  This implies that $F_T(x)$ is a rational function. Thus, the twisted trace $T$ is degenerate, contradicting our assumption.  
\end{proof}

However, we can give an interpretation of this divergent series using the Borel transform. Recall that the Borel transform (see \cite[\S 4.4]{costinAsymptoticsBorel2009}) is the transformation of power series $\mathcal{B}\colon x^{-1}\C[[x^{-1}]]\to \C[[s]]$ given by:
\[\mathcal{B}(\sum_{n=0}^{\infty}a_nx^{-n-1})=\sum_{n=0}^\infty \frac{a_n}{n!}s^n.\]

It is a well-known result (often called ``frequency shifting'' in the context of the Laplace transform) that:
\begin{multline*}
\mathcal{B}\left(\sum_{n=0}^{\infty}a_n(x+b)^{-n-1}\right)=\mathcal{B}\left(\sum_{n=0}^{\infty}\sum_{m=0}^{\infty} a_n\binom{n+m+1}{m}\frac{(-b)^m}{x^{n+m+1}}\right)\\=\sum_{n=0}^{\infty}\sum_{m=0}^{\infty} a_n\frac{(-b)^ms^{n+m}}{m!(n+1)!}=\left(\sum_{m=0}^{\infty}\frac{(-bs)^{m}}{m!}\right)\left(\sum_{n=0}^{\infty} a_n\frac{s^{n}}{(n+1)!}\right)=e^{-bs}\mathcal{B}\left(\sum_{n=0}^{\infty}a_nx^{-n-1}\right)
\end{multline*}
Applying this to the equation \eqref{eq:P-principal}, we find that 
\begin{align}\notag
(e^{-s/2}-te^{s/2})	\mathcal{B}(F_T)&=\sum_{a\in \C} {e^{as}}\sum_{k=1}^{p_a}\frac{D_a^{(k)}}{(k-1)!}s^{k-1}\\
	\mathcal{B}(F_T)&=\sum_{a\in \C} \frac{e^{as}}{e^{-s/2}-te^{s/2}}\sum_{k=1}^{p_a}\frac{D_a^{(k)}}{(k-1)!}s^{k-1}\label{eq:Borel-transform}
\end{align}
The right-hand side can only have poles at solutions to $t=e^{-s}$; of course, all these solutions have $\Re (s)=-\log|t|$.  

From this Borel transform, we can construct meromorphic solutions to this recursion by taking the Laplace transform along rays.    For $\phi\in [0,2\pi]$ we say that the ray through $e^{i\phi}$ is a {\bf Stokes ray}
if there is a solution of $t=e^{-s}$ for $\arg(s)=\phi$.   If we write $t=|t|e^{i\theta}$, these Stokes rays are the rays from the origin through the points $-\log |t| +i(\theta+2\pi k)$ for $k\in \Z$.  If $\cos\phi$ has the same sign at $\log |t|$, then $\phi$ cannot give a Stokes ray; put differently, every nonzero point on a Stokes ray has real part with the same sign as $-\log |t|$.  

If $\phi$ does not correspond to a Stokes ray,
we define the function:
\begin{equation}\label{eq:phi-transform}
    \tilde{F}_T^{\phi}(x)=\int_{0}^{\infty}e^{i\phi}\mathcal{B}(F_T)(s e^{i\phi})e^{-xe^{i\phi}s}ds.
\end{equation}
Let $\rho_P=\{a\in \C\mid P(a)=0\}$ be the roots of $P$ and $\rho^{\pm}_P=\{a\pm \half\pm\Z_{\geq 0}\}$.
Let $H_{\phi}=\{x\in \C\mid \Re(xe^{i\phi})>\Re((a\pm \half)e^{i\phi})\text{ for all }a\in \rho_P\}$.
\begin{thm}
    The integral $\tilde{F}_T^{\phi}(x)$ converges locally uniformly on $H_{\phi}$ to a holomorphic solution of 
    \begin{equation}\label{eq:tilde-recursion}
    \tilde{F}_T^{\phi}(x+\half)-t\tilde{F}_T^{\phi}(x-\half) =\frac{Q(x)}{P(x)}=\sum_{k=1}^{\infty}D^{(k)}(x)
    \end{equation}
    which can be meromorphically continued to the entirety of $\C$ with all poles in the set $\rho^{\mp}_P$ if $\pm \cos\phi>0$.   As a function of $\phi$ on $\mathbb R/2\pi \Z$, the function $\tilde{F}_T^{\phi}(x)$ is locally constant, with the value changing only at the argument of Stokes rays.  
\end{thm}
\begin{proof}
  For simplicity, assume that $\cos \phi\geq 0$; the case $\cos \phi\leq 0$ is easily dealt with by reversing some signs.   
    Most parts of this are standard results of Borel-Laplace theory.  In particular, if $x\in H_{\phi}$, then the integrand of \eqref{eq:phi-transform} can be written as \[ \sum_{a\in \C} e^{(-x +a+\half )e^{i\phi}s}\frac{e^{i\phi}}{e^{-e^{i\phi}s}-t}\sum_{k=1}^{p_a}\frac{D_a^{(k)}}{(k-1)!}(se^{i\phi})^{k-1}\] 
    The first exponential in each summand is decaying since $x\in H_{\phi}$, and the remaining terms are bounded above by a polynomial. 
    Thus, the function $ \tilde{F}_T^{\phi}$ is holomorphic on $H_{\phi}$ (see, for example, \cite[Prop. 2.7]{costinAsymptoticsBorel2009}).  The difference $\tilde{F}_T^{\phi}-\tilde{F}_T^{\phi'}$ can be written as a contour integral that surrounds all poles with argument between $\phi$ and $\phi'$, so if there are no Stokes rays in this region, the difference will be 0.  
    
    The recursion \eqref{eq:tilde-recursion} holds if $x\pm \half \in H_{\phi}$ because 
    \begin{align*}
        \tilde{F}_T^{\phi}(x+\half)-t\tilde{F}_T^{\phi}(x-\half) &=\int_{0}^{\infty}e^{i\phi}\bigg(\mathcal{B}(F_T)((s+\half) e^{i\phi})e^{-xe^{i\phi}(s+\half)}-\mathcal{B}(F_T)((s-\half) e^{i\phi})e^{-xe^{i\phi}(s-\half)}\bigg)ds\\
        &= \int_{0}^{\infty}e^{i\phi}e^{-xe^{i\phi}s}\sum_{a\in \C} e^{as}\sum_{k=1}^{p_a}\frac{D_a^{(k)}}{(k-1)!}s^{k-1}ds\\
        &=\int_{0}^{\infty}e^{i\phi}e^{-xe^{i\phi}s}\mathcal{B}\Big(\frac{Q(x)}{P(x)}\Big)ds\\
        &=\frac{Q(x)}{P(x)}
    \end{align*}
    where the last equality follows from the fact that if the Borel transform $\mathcal{B}(f)$ converges to an entire function, then $f$ converges in a neighborhood of $\infty$ to the Laplace transform along any ray; in particular, on $H_{\phi}$, applying Laplace transform termwise results in the Taylor series of $\frac{Q(x)}{P(x)}$ at $\infty$, which is uniformly convergent on $H_{\phi}$.   

    Finally, we can meromorphically continue $\tilde{F}_T^{\phi}$ to the rest of the complex plane by using \eqref{eq:tilde-recursion} as the definition---for every $x\in \C$, we have $x+n\in H_{\phi}$ for some $n\in \Z_{\geq 0}$, and so we can define
    \[\tilde{F}_T^{\phi}(x):=\tilde{F}_T^{\phi}(x+n)-\sum_{i=0}^{n-1} \frac{Q(x+i+\half)}{P(x+i+\half)}.\]
    This defines a holomorphic function on an open disk centered on $x$ if and only if the rational function $\frac{Q(x)}{P(x)} $ has no poles in $\{x+\half,\dots, x+n-\half\}$, that is, if $x\notin \rho^-_P$.  
\end{proof}

In some cases, we can write the functions $\tilde{F}^{\phi}_T$ in terms of well-known holomorphic functions.  Consider the {\bf Lerch transcendent} function
\[\Phi(t,n,x)=\sum_{j=0}^{\infty} \frac{t^j}{(x+j)^n}.\]  While frequently in the literature, this is considered as a meromorphic function of $n$, we will only consider the case where $n$ is a positive integer; this sum absolutely converges and thus defines a meromorphic function on the open domain $\C\setminus \Z_{\leq 0}$ if $|t|\leq 1$ and $n\geq 1$ and one of these inequalities is strict. If $t\notin [1,\infty)$, then the integral formula 
\begin{equation}\label{eq:lerch-integral}
\Phi(t, n, x)=\frac{1}{n!} \int_0^{\infty} \frac{s^{n-1} e^{-x s}}{1-t e^{-s}} d s  
\end{equation}
defines this function on the half-plane $\operatorname{Re}(x)>0$.   This function meromorphically extends to the rest of the complex plane with poles at $x\in \Z_{\leq 0}$ via the equality $\Phi(t,n,x)-t\Phi(t,n,x+1)=\frac{1}{x^n}$.  Of course, the integral formula is precisely the Laplace transform along the positive real axis of the function $ \frac{s^{n-1}}{(n-1)!\cdot (1-te^{-s})}$, which is the Borel transform of the unique power series that solves the recursion above.

We can use this function to describe the function $F^{\phi}_T$ when $\cos\phi$ has the same sign as $\log |t|$, that is, $\cos\phi>0$ if $|t|>1$ and $\cos \phi<0$ if $|t|<1$.  Consider the functions \[\lerchp{n}{a}{x}=-t^{-1}\Phi(t^{-1},n,x-a+\half) \qquad \lerchm{n}{a}{x}= (-1)^n\Phi(t,n,a-x+\half).\]  These functions satisfy the recursion
\begin{equation}
        \lerchp{n}{a}{x+\half}-t\lerchp{n}{a}{x-\half}=\lerchm{n}{a}{x}-t \lerchm{n}{a}{x}= \frac{1}{(x-a)^n}.
\end{equation}

Throughout the result below, we assume $t\neq 1$ and that \eqref{eq:P-principal} gives the principal part decomposition of $\frac{Q(z)}{P(z)}$. 
\begin{thm}\label{thm:lerch-expansion}
     If $t\notin (0,1]$ then:
    \begin{equation}\label{eq:lerch-FT}
        \tilde{F}^{0}_T(x)=\sum_{a\in \C} \sum_{\ell=1}^{p_a}D_{a}^{(\ell)}\lerchp{\ell}{a}{x}.
        \end{equation}
If $|t|\geq 1$ and $\cos \phi>0$ then $\tilde{F}^{\phi}_T(x)=\tilde{F}^{0}_T(x) $.  Similarly, if $t\notin [1,\infty)$ then:       
        \begin{equation}\tilde{F}^{\pi}_T(x)=\sum_{a\in \C} \sum_{\ell=1}^{p_a}D_{a}^{(\ell)}\lerchm{\ell}{a}{x}.        
    \end{equation}
    If $|t|\leq 1$ and $\cos\phi <0$ then $\tilde{F}^{\phi}_T(x)=\tilde{F}^{\pi}_T(x) $.
\end{thm}
On the other hand, if $\phi$ is separated from the positive and negative real axes by a Stokes ray, we don't have a similar infinite series expansion of $F^{\phi}$ but it has a similar expansion in terms of Laplace transforms of $ \frac{s^{n-1}}{(n-1)!\cdot (1-te^{-s})}$ along the rays with $\frac{\theta+2\pi k}{-\log |t|}\leq \tan \phi\leq \frac{\theta+2\pi (k+1)}{-\log |t|}$.  

If $t=1$, then the situation is more subtle.  We can still calculate what the Borel transform of our solution should be, but this transform could have a pole at the origin, in which case the result is not the Borel transform of a power series in $x^{-1}\C[[x^{-1}]]$. 
This pole is simple if it exists and the residue at $s=0$ is $-t^{-1}\sum_{a\in \C}D_a^{(1)}$.

Thus, if $\sum_{a\in \C}D_a^{(1)}=0$, then the numerator and denominator of \eqref{eq:Borel-transform} both have simple zeros at $s=0$ which cancel, and so we can use \eqref{eq:phi-transform} to define solutions to the recursion \eqref{eq:tilde-recursion}.  In this case, the only Stokes rays are the two halves of the imaginary axis, so the solution  $\tilde{F}_T^{\phi}$ only depends on whether $\cos\phi$ is positive or negative.

{\renewcommand{\markboth}[2]{}\printbibliography}
\end{document}